\documentclass[11pt]{article}

\usepackage{amssymb,amsthm,amsmath}

\newcommand{\dd}{\mathrm{d}}
\newcommand{\E}{\mathbb{E}}
\newcommand{\1}{\textbf{1}}
\newcommand{\R}{\mathbb{R}}

\newcommand{\p}[1]{\mathbb{P}\left( #1 \right)}

\usepackage[paper=a4paper, left=1.3in, right=1.3in, top=1.2in, bottom=1.2in]{geometry}
\newtheorem{theorem}{Theorem}
\newtheorem{lemma}[theorem]{Lemma}
\newtheorem{corollary}[theorem]{Corollary}

\theoremstyle{remark}

\theoremstyle{definition}

\begin{document}

\begin{center}
\begin{large}
STOCHASTIC DOMINANCE AND WEAK CONCENTRATION FOR SUMS OF INDEPENDENT SYMMETRIC RANDOM VECTORS\\
\end{large}
Witold Bednorz\footnote{University of Warsaw} and Tomasz Tkocz\footnote{Carnegie Mellon University}
\end{center}

\bigskip

\begin{footnotesize}

\begin{center}
\parbox{0.75\textwidth}{
\noindent\textbf{Abstract.}
Kwapie\'n and Woyczy\'nski asked in their monograph (1992) whether their notion of superstrong domination is inherited when taking sums of independent symmetric random vectors (one vector dominates another if, essentially, tail probabilities of any norm of the two vectors compare up to some scaling constants). We answer this question positively. As a by-product of our methods, we establish that a certain notion of weak concentration is also preserved by taking sums of independent symmetric random vectors.
}
\end{center}

\bigskip


\end{footnotesize}

\bigskip

\section{Introduction}

Stochastic orderings quantitatively capture the notion of one random variable being \emph{greater} than another one. Common examples include $\mathcal{U}$-stochastic orderings. If $\mathcal{U}$ is a family of real valued functions defined on, say a real separable Banach space $E$, we say for $E$-valued random vectors $X$ and $Y$ that $X$ is $\mathcal{U}$-dominated by $Y$, written $X \prec_{\mathcal{U}} Y$, if $\E f(X) \leq \E f(Y)$, for all functions $f$ in $\mathcal{U}$. For instance, taking $\mathcal{U}$ to be the family of nonnegative convex functions on $E$ results in the usual convex stochastic ordering, or considering $\mathcal{U}$ to be the family of the exponents of bounded linear functionals, $\exp\{x^*(\cdot)\}$, $x^* \in E^*$, can be used to define sub-Gaussian random vectors, just to name several important examples. An inductive argument shows that these two orderings are inherited for sums: if $X_i \prec_{\mathcal{U}} Y_i$, for $i = 1,\ldots,n$ and the $X_i$ and $Y_i$ are independent, then $\sum_{i=1}^n X_i \prec_{\mathcal{U}}  \sum_{i=1}^n Y_i$. This is a significant property of a stochastic ordering as it allows to compare sums in presence of the comparison for independent summands. The main goal of this article is to establish such a \emph{tensorisation} property for symmetric random vectors of a stochastic ordering called \emph{superstrong domination}, which we shall now define.

Let $X$ and $Y$ be symmetric random vectors with values in a real separable Banach space $E$. We say in this paper that $Y$ dominates $X$ with constants $\kappa, \lambda \geq 1$ ($(\kappa,\lambda)$-dominates, in short) if for every closed convex and symmetric set $K$ in $E$ we have
\[
\p{X \notin K} \leq \kappa\p{\lambda Y \notin K}.\]
We sometimes write $X \prec_{(\kappa,\lambda)} Y$. Equivalently, $X$ is $(\kappa,\lambda)$-dominated by $Y$ if for every continuous norm $\|\cdot\|$ on $E$ we have
\[
\p{\|X\|>1} \leq \kappa\p{\lambda\|Y\|>1}.\]
(The inequality for convex sets clearly implies the inequality for norms. Conversely, given a closed convex and symmetric set $K$ in $E$, take $K_{\delta,R}$ to be the $\delta$-enlargement of $K$, $\{x \in E, \textrm{dist}(x,K) \leq \delta\}$ intersected with the closed ball of radius $R$ in $E$. The Minkowski functional of $K_{\delta,R}$ defines a continuous norm on $E$. Letting $\delta$ go to $0$ and $R$ to $\infty$ finishes the argument.)

This notion appears as \emph{superstrong domination} in the monograph by Kwapie\'n and Woyczy\'nski (see \cite{KW}, Chapters 3.2 and 3.6). It can be viewed as a less restrictive version of the $\mathcal{U}$-stochastic ordering for the family $\mathcal{U}$ comprising the indicators of complements of convex symmetric sets. The special case, $\kappa = \lambda = 1$ and $E = \R^d$ is known as \emph{peakedness} and was first introduced by Birnbaum (univariate case, $d = 1$, see \cite{Bir}), Sherman (multivariate case, $d \geq 1$, see \cite{Sher}) and by Kanter (see \cite{K}). In this case, by considering symmetric strips, if $X \prec_{(1,1)} Y$, then necessarily $\E|\langle t, X \rangle|^2 \leq \E|\langle t, Y \rangle|^2$, for all vectors $t$ in $\R^d$. For symmetric Gaussian random vectors $X$ and $Y$, this simple necessary condition is also sufficient! (Since the matrix $[\E(Y_iY_j-X_iX_j)]_{i,j}$ is positive semi-definite, there is an independent symmetric Gaussian random vector $Z$ such that $Y = X+Z$ and Anderson's inequality, see \cite{And} finishes the argument.)

Kwapie\'n and Woyczy\'nski posed a question whether superstrong domination is preserved by taking sums of independent symmetric random vectors. They remarked that the answer is positive for vectors taking values in one-dimensional subspaces, crediting this result to Jain and Marcus (see \cite{JM} and Theorem 3.2.1 in \cite{KW}). Kanter's result says that the peakedness of log-concave measures in $\R^d$ tensorises (see Corollary 3.2 in \cite{K}), supporting the ``yes" answer in this case. Our first main result provides the positive answer in full generality.

\begin{theorem}\label{thm:kwaptensor}
Let $X_1,\ldots,X_n$ and $Y_1,\ldots,Y_n$ be independent symmetric random vectors with values in separable Banach space. Suppose that $X_i$ is $(\kappa,\lambda)$-dominated by $Y_i$ for each $i = 1,\ldots,n$. Then the sum $X_1+\ldots+X_n$ is $(16\alpha^{-1}\lceil \kappa \rceil, (1+\alpha)\lceil \kappa \rceil \lambda)$-dominated by the sum $Y_1+\ldots+Y_n$ for any $0 < \alpha \leq 1$. 
\end{theorem}

The crux of our argument is to devise a \emph{proxy}, a quantity that mimics tail probabilities $\p{\|X\|>1}$, but, as opposed to them, gives rise to inequalities that are easy to tensorise. With the aid of the proxy as well as several tools for random signs, we first establish the tensorisation of $(1,1)$-domination. Then we show how to deduce the theorem for arbitrary $\kappa$ and $\lambda$ from the case $\kappa = 1 = \lambda$.

Using similar tools and techniques we derive a tensorisation property for a certain notion of weak concentration, which can be of independent interest. We say that a symmetric random vector $X$ with values in a separable Banach space $E$ satisfies the \emph{weak Borell} inequality (or the weak concentration) with constants $C \geq 1$, $\delta > 0$ and $0 < \theta < 1$ ($\textrm{WB}(C,\delta,\theta)$ for short), if for every continuous norm $\|\cdot\|$ on $E$ such that $\p{\|X\|>1} < \theta$, we have
\begin{equation}\label{eq:WBdef}\tag{WB}
\p{\|X\|>\lambda} \leq C\lambda^{-\delta}\p{\|X\| >1}, \qquad \lambda \geq 1.
\end{equation}
Again, this is the same as saying that for $\lambda \geq 1$ and every closed symmetric convex subset $K$ of $E$ such that $\p{X \notin K} < \theta$, we have $\p{X \notin \lambda K} \leq C\lambda^{-\delta}\p{X \notin K}$.

For instance, if $X$ is a $\kappa$-concave random vector in $\R^d$, $\kappa < 0$ and $\theta < 1/2$, then $X$ satisfies $\text{WB}(C,-1/\kappa,\theta)$ with $C$ dependent only on $\kappa$ and $\theta$ (see \cite{Bor-mea} and \cite{Bor-fun}). Even though in general $\kappa$-concavity with negative $\kappa$ is not preserved by taking sums of independent vectors, the WB inequality is (modulo a slight change of constants), which is our second main result.

\begin{theorem}\label{thm:WBtensor}
Suppose that symmetric random vectors $X_1,\ldots,X_n$ in separable Banach space are independent and each satisfies $\text{WB}(C,\delta,\theta)$. Then $X_1 + \ldots + X_n$ satisfies $\text{WB}(C',\delta,\theta')$, where $C' = 12\cdot 9^{\delta}C$ and $\theta' = \min\{\theta/2,(96C\cdot 9^\delta)^{-1}\}$.
\end{theorem}

As an application of this result, we establish superstrong domination for weighted sums of i.i.d. symmetric random vectors satisfying the weak Borell inequality when the sequences of weights are comparable in terms of majorisation. Recall that one sequence of real numbers $a=(a_1,\ldots,a_n)$ is majorised by another one $b=(b_1,\ldots,b_n)$, usually denoted $a\prec b$, if the nonincreasing rearrangements $a_1^*\geq\ldots\geq a_n^*$ and $b_1^*\geq\ldots\geq b_n^*$ of $a$ and $b$ satisfy the inequalities
\[
\sum_{j=1}^k a^*_j \leq \sum_{j=1}^k b_j^* \ \ \mbox{for each } 1\leq k\leq n-1 \ \ \mbox{and } \ \sum_{j=1}^n a_j = \sum_{j=1}^n b_j.\]
Equivalently, $a$ is a convex combination of the permutations $(b_1',\ldots,b_n')$ of $b$ (see for example Theorem II.1.10 in \cite{Bh}).

\begin{theorem}\label{thm:schur}
Let $X_1, X_2,\ldots$ be i.i.d. symmetric random vectors in separable Banach space. Assume that  for some $C > 0$, $0 < \theta < 1$ and $\delta > 1$ each $X_i$ satisfies $\text{WB}(C,\delta,\theta)$. Let $a = (a_1,\ldots,a_n)$ and $b = (b_1,\ldots,b_n)$ be sequences of real numbers such that $a$ is majorised by $b$. Then
\[
\sum_{i=1}^n a_iX_i \prec_{(\kappa,\lambda)} \sum_{i=1}^n b_iX_i
\]
with $\kappa = \max\{2\theta^{-1},96C\cdot 9^\delta,12C\cdot 9^{\delta}(\delta-1)^{-1}\}$ and $\lambda = 2$.
\end{theorem}

This theorem does not hold under the weaker assumption that the $X_i$ satisfy the weak concentration with $\delta < 1$. To see that, fix $\delta \in (0,1)$ and take $X_i$ to be independent real valued symmetric $\delta$-stable random variables. Then $\p{|X_1| > t} \sim t^{-\delta}$, for large $t$, hence the $X_i$ satisfy $\text{WB}(C,\delta',\theta)$ if and only if $\delta' \leq \delta$. Consider the sequences $a = (1/n,\ldots,1/n)$ and $b = (1,0,\ldots,0)$. Then $a \prec b$ and $\sum a_iX_i$ has the same distribution as $n^{1/\delta-1}X_1$, so $\sum_{i=1}^n a_iX_i \prec_{(\kappa,\lambda)} \sum_{i=1}^n b_iX_i$ would particularly imply that $\p{ |X_1| > 1 } \leq \kappa \p{\lambda |X_1| > n^{1/\delta-1}}$, which is not true for large $n$. We suspect that our assumption of $\delta > 1$ can be weakened to $\delta \geq 1$.

\section{Auxiliary results}\label{sec:auxres}
 
In this section we collect several well-known inequalities which will be needed in our proofs. We begin with three results for random signs. Here and throughout $\varepsilon_1, \varepsilon_2, \ldots$ are independent random signs each taking the value $\pm 1$ with probability $1/2$. Let $v_1,\ldots,v_n$ be vectors in a separable Banach space $(E,\|\cdot\|)$. Kahane's inequality (see~\cite{Kah} or Proposition 1.4.1. in \cite{KW}) says that for $s, t > 0$, we have
\begin{equation}\label{eq:kah}
\p{\|{\sum} \varepsilon_i v_i\| > s + t} \leq 4\p{\|{\sum} \varepsilon_i v_i\| > s}\p{\|{\sum} \varepsilon_i v_i\| > t}.
\end{equation}
We also recall the optimal $L_1-L_2$ moment comparison due to Lata\l a and Oleszkiewicz (see \cite{LO}), that is
\begin{equation}\label{eq:LO}
\E\|{\sum} \varepsilon_i v_i\|^2 \leq 2(\E\|{\sum} \varepsilon_i v_i\|)^2.
\end{equation}
This, combined with the Paley-Zygmund inequality yields that for any $\theta \in (0,1)$, we have
\begin{equation}\label{eq:LOPZ}
\p{\|{\sum} \varepsilon_i v_i\| > \theta \E\|{\sum} \varepsilon_i v_i\|} \geq \frac{1}{2}(1-\theta)^2.
\end{equation}
The contraction principle (see for instance Theorem 4.4 in \cite{LT}) in particular asserts that for two sequences of real numbers $(a_i)_{i=1}^n$ and $(b_i)_{i=1}^n$ such that $|a_i| \leq |b_i|$ for each $i \leq n$, we have
\begin{equation}\label{eq:contr}
\E\|{\sum} \varepsilon_i a_i v_i\| \leq \E\|{\sum} \varepsilon_i b_i v_i\|.
\end{equation}

Let us recall several classical inequalities for sums of symmetric independent random vectors $X_1,\ldots,X_n$ with values in the separable Banach space $(E,\|\cdot\|)$. Denote as usual $S_j = X_1+\ldots+X_j$, $j \leq n$, $X_n^* = \max_{j\leq n}\|X_j\|$ and $S_n^* = \max_{j\leq n}\|S_j\|$. The L\'evy inequality says that
\begin{equation}\label{eq:Levy}
\p{S_n^* > t} \leq 2\p{\|S_n\|>t}, \qquad t \geq 0.
\end{equation}
Moreover, we have
\begin{equation}\label{eq:Xn*}
\p{X_n^*>t} \leq 2\p{\|S_n\|>t}, \qquad t \geq 0.
\end{equation}
The Hoffmann-J\o rgensen inequality asserts that
\begin{equation}\label{eq:HJ}
\p{S_n^* > s+t+u} \leq \p{X_n^* >s} + 2\p{S_n^*>t}\p{\|S_n\|>u}, \qquad s,t,u \geq 0.
\end{equation}
Lastly, even without the symmetry of the $X_i$, we have
\begin{equation}\label{eq:sumX_iuppbd}
\sum_{j=1}^n\p{\|X_j\|>t} \leq \frac{\p{X_n^*>t}}{1-\p{X_n^*>t}}, \qquad t \geq 0.
\end{equation}
(All of these inequalities can be found for instance in Chapter 1 of \cite{KW}).

\section{Proof of Theorem \ref{thm:kwaptensor}}

The goal of the first three subsections is to show Theorem \ref{thm:kwaptensor} when $\kappa = 1 = \lambda$. In the last subsection we show how to deduce the general case.

\subsection{Conditional convexity and a proxy}

We start with a simple lemma which lies at the heart of our tensorisation argument.

\begin{lemma}\label{lm:condconvex}
Let $X_1,\ldots, X_n$ and $Y_1,\ldots,Y_n$ be independent symmetric random vectors with values in a separable Banach space $E$ such that $X_i$ is $(1,1)$-dominated by $Y_i$ for each $i = 1,\ldots,n$. Let $\varphi\colon E^n \rightarrow [0,\infty)$ be a continuous function, convex with respect to each coordinate. Then for $t \geq 0$ we have
\[
\p{\E_\varepsilon \varphi(\varepsilon_1X_1,\ldots,\varepsilon_nX_n) > t} \leq \p{\E_\varepsilon \varphi(\varepsilon_1Y_1,\ldots,\varepsilon_nY_n) > t}.
\]
\end{lemma}
\begin{proof}
We condition on $X_2,\ldots,X_n$ and define the set
\[
K = \{x \in E, \ \E_\varepsilon\varphi(\varepsilon_1x,\varepsilon_2X_2,\ldots,\varepsilon_nX_n) \leq t\},
\] 
which is closed, convex and symmetric. Using $X_1 \prec_{(1,1)} Y_1$, we get $\p{X_1 \notin K} \leq \p{Y_1 \notin K}$ which means that
\[
\mathbb{P}_{X_1}\left(\E_\varepsilon \varphi(\varepsilon_1X_1,\varepsilon_2X_2,\ldots,\varepsilon_nX_n) > t\right) \leq \mathbb{P}_{Y_1}\left(\E_\varepsilon \varphi(\varepsilon_1Y_1,\varepsilon_2X_2,\ldots,\varepsilon_nX_n) > t\right),
\]
so taking the expectation of both sided against $X_2,\ldots,X_n$ gives
\[
\p{\E_\varepsilon \varphi(\varepsilon_1X_1,\varepsilon_2X_2,\ldots,\varepsilon_nX_n) > t} \leq \p{\E_\varepsilon \varphi(\varepsilon_1Y_1,\varepsilon_2X_2,\ldots,\varepsilon_nX_n) > t}.
\]
Similarly, we condition on $Y_1,X_3,\ldots,X_n$ to swap $X_2$ for $Y_2$, etc. and finally arrive at the desired inequality.
\end{proof}

Note that the function $u \mapsto (u-1)_+ = \max\{u-1,0\}$ is convex and nondecreasing. Therefore, for a normed space $(E,\|\cdot\|)$ the function $\varphi\colon E^n\rightarrow [0,\infty)$ defined by $\varphi(x_1,\ldots,x_n) = (\|\sum_{i=1}^n x_i\|-1)_+$ is convex (and continuous). From Lemma \ref{lm:condconvex} we thus get the following corollary.
\begin{corollary}\label{cor:proxy}
Let $X_1,\ldots, X_n$ and $Y_1,\ldots,Y_n$ be independent symmetric random vectors with values in a separable Banach space $E$ such that $X_i$ is $(1,1)$-dominated by $Y_i$ for each $i = 1,\ldots,n$. Let $\|\cdot\|$ be a continuous norm on $E$. Then for $t \geq 0$ we have
\[
\p{\E_\varepsilon (\|{\sum} \varepsilon_iX_i\|-1)_+ > t } \leq \p{\E_\varepsilon (\|{\sum} \varepsilon_iY_i\|-1)_+ > t }.
\]
In particular,
\begin{equation}\label{eq:tens}
\int_0^1 \p{\E_\varepsilon (\|{\sum} \varepsilon_iX_i\|-1)_+ > t } \dd t \leq \int_0^1 \p{\E_\varepsilon (\|{\sum} \varepsilon_iY_i\|-1)_+ > t } \dd t.
\end{equation}
\end{corollary}
For a nonnegative random variable $Y$ we plainly have
\[
\E \min\{Y,1\} = \int_0^\infty \p{\min\{Y,1\}>t} \dd t = \int_0^\infty \p{Y > t, 1 > t} \dd t = \int_0^1 \p{Y>t} \dd t.
\]
Therefore, in view of this corollary, the following quantity
\[
\int_0^1 \p{\E_\varepsilon (\|{\sum} \varepsilon_iX_i\|-1)_+ > t } \dd t = \E\min\{\E_\varepsilon (\|{\sum} \varepsilon_iX_i\|-1)_+,1\} 
\]
tensorises as well. This is our proxy and we will show that it is comparable to $\p{\|\sum X_i\| > 1}$. This is where the assumption of symmetry and the aforementioned tools for random signs come into play.

\subsection{Upper and lower bounds for the proxy}

\begin{lemma}\label{lm:lowbd}
Suppose that $X_1,\ldots,X_n$ are independent symmetric random vectors in a normed space $(E,\|\cdot\|)$. Then for $0 < \alpha \leq 1$ we have
\[
\E\min\{\E_\varepsilon (\|{\sum} \varepsilon_iX_i\|-1)_+,1\} \geq \alpha\p{\|{\sum} X_i\| > 1 + \alpha}.
\]
\end{lemma}
\begin{proof}
Denote $U = \E_\varepsilon(\|{\sum} \varepsilon_iX_i\|-1)_+$. Notice that for a positive parameter $\alpha$ we have 
\[
U \geq \E_\varepsilon (\|{\sum} \varepsilon_iX_i\|-1)_+\1_{\{\|\sum \varepsilon_i X_i\| > 1+ \alpha\}} \geq \alpha \mathbb{P}_\varepsilon\big(\|{\sum} \varepsilon_iX_i\| > 1 + \alpha\big).
\]
Thus,
\[
\E\min\{U,1\} \geq \E\min\{\alpha \mathbb{P}_\varepsilon\big(\|{\sum} \varepsilon_iX_i\| > 1 + \alpha\big),1\}.
\]
When $\alpha \leq 1$ the last expression becomes $\alpha\p{\|{\sum} X_i\| > 1 + \alpha}$.
\end{proof}

\begin{lemma}\label{lm:upbd}
Let $X_1,\ldots,X_n$ be independent symmetric random vectors in a normed space $(E,\|\cdot\|)$. Then we have
\[
\E\min\{\E_\varepsilon (\|{\sum} \varepsilon_iX_i\|-1)_+,1\} \leq 16\p{\|{\sum} X_i\| > 1}.
\]
\end{lemma}
\begin{proof}
For $p \in (0,1)$ define the event 
\[
A_p = \{\mathbb{P}_\varepsilon\big(\|{\sum} \varepsilon_iX_i\| > 1\big)  >p\}.
\]
Clearly, we have
\[
\E\min\{\E_\varepsilon (\|{\sum} \varepsilon_iX_i\|-1)_+,1\} \leq \E\1_{A_p} + \E \E_\varepsilon (\|{\sum} \varepsilon_iX_i\|-1)_+ \1_{A_p^c}.
\]
We handle the first term directly by Markov's inequality,
\[
\p{A_p} \leq \frac{1}{p}\E\mathbb{P}_\varepsilon\big(\|{{\sum}} \varepsilon_iX_i\| > 1\big) = \frac{1}{p} \p{\|{{\sum}} X_i\| > 1}.
\]
To deal with the second term, first notice that by Kahane's inequality \eqref{eq:kah} we have
\[
\E_\varepsilon (\|{\sum} \varepsilon_iX_i\|-1)_+ = \int_0^\infty \mathbb{P}_\varepsilon\big(\|{\sum} \varepsilon_iX_i\| > 1+t\big) \dd t \leq 4\mathbb{P}_\varepsilon\big(\|{\sum} \varepsilon_iX_i\| > 1\big)\E_\varepsilon\|{\sum} \varepsilon_iX_i\|.
\]
Second, notice that the quantity $\E_\varepsilon\|{\sum} \varepsilon_iX_i\|$ is bounded on the event $A_p^c$. Indeed, suppose that $\E_\varepsilon\|{\sum} \varepsilon_iX_i\| >1$ and set $\theta = (\E_\varepsilon\|{\sum} \varepsilon_iX_i\|)^{-1}$. Then on $A_p^c$, by \eqref{eq:LOPZ},
\[
p \geq \mathbb{P}_\varepsilon\big(\|{\sum} \varepsilon_iX_i\| > 1\big) = \mathbb{P}_\varepsilon\big(\|{\sum} \varepsilon_iX_i\| > \theta\E_\varepsilon\|{\sum} \varepsilon_iX_i\|\big) \geq \frac{1}{2}(1-\theta)^2,
\]
so $\theta \geq 1 - \sqrt{2p}$ and provided that $p < 1/2$, we get
\[
\E_\varepsilon\|{\sum} \varepsilon_iX_i\| \leq \frac{1}{1 - \sqrt{2p}}.
\]
Putting these together yields
\[
\E \E_\varepsilon (\|{\sum} \varepsilon_iX_i\|-1)_+ \1_{A_p^c} \leq \frac{4}{1-\sqrt{2p}}\E \mathbb{P}_\varepsilon\big(\|{\sum} \varepsilon_iX_i\| > 1\big) = \frac{4}{1-\sqrt{2p}}\p{\|{{\sum}} X_i\| > 1}.
\]
Altogether,
\[
\E\min\{\E_\varepsilon (\|{\sum} \varepsilon_iX_i\|-1)_+,1\} \leq \left(\frac{1}{p} + \frac{4}{1-\sqrt{2p}}\right)\p{\|{{\sum}} X_i\| > 1}.
\]
Choosing $p = 1/8$ finishes the proof (the optimal choice $p \approx 0.16$ gives the constant $\approx 15.45$).
\end{proof}

\subsection{Proof in the case $\kappa = 1 = \lambda$}

Suppose that $X_1,\ldots, X_n$ and $Y_1,\ldots,Y_n$ are independent symmetric random vectors with values in a separable Banach space $E$. Let $X_i$ be $(1,1)$-dominated by $Y_i$ for each $i \leq n$. Fix a continuous norm $\|\cdot\|$ on $E$. We would like to show that $\p{\|{\sum} X_i\| > 1} \leq \kappa
\p{\lambda\|{\sum} Y_i\| > 1}$ for some universal constants $\kappa$ and $\lambda$. Fix $0 < \alpha \leq 1$. Applying consecutively Lemma \ref{lm:lowbd}, Corollary \ref{cor:proxy} and Lemma \ref{lm:upbd} yields
\[
\alpha \p{\|{\sum} X_i\| > 1+\alpha} \leq 16 \p{\|{\sum} Y_i\| > 1}.
\]
Rescaling the norm gives the desired bound with $\kappa = 16\alpha^{-1}$ and $\lambda = 1+\alpha$. \hfill$\square$

\subsection{Reduction to the case $\kappa = 1 = \lambda$}

We describe two arguments leading to the conclusion that it suffices to prove Theorem~\ref{thm:kwaptensor} when $\kappa = 1 = \lambda$, thus finishing the whole proof. 

The first argument is based on the following lemma whose proof is essentially given in the second step of the proof of Theorem 3.2.1 in \cite{KW}. We sketch it for completeness.

\begin{lemma}\label{lm:reduction}
Suppose that for every $n\geq 1$ and independent symmetric random vectors $X_1,\ldots,X_n$ and $Y_1,\ldots,Y_n$ in separable Banach space, the following is true

``If $X_i$ is $(1,1)$-dominated by $Y_i$, $i\leq n$, then $\sum X_i$ is $(\kappa_0,\lambda_0)$-dominated by $\sum Y_i$."

\noindent
Then for every $\kappa, \lambda \geq 1$, $n\geq 1$ and independent symmetric random vectors $X_1,\ldots,X_n$ and $Y_1,\ldots,Y_n$ in separable Banach space such that $X_i$ is $(\kappa,\lambda)$-dominated by $Y_i$, $i\leq n$, we have that $\sum X_i$ is $(\lceil\kappa\rceil\kappa_0,\lceil\kappa\rceil\lambda\lambda_0)$-dominated by $\sum Y_i$.
\end{lemma}
\begin{proof}
Suppose that $X_i \prec_{(\kappa,\lambda)} Y_i$. The main idea is to take auxiliary random variables $\delta_{i,k}$, $i \leq n$, $k \leq \lceil\kappa\rceil$, independent of the $X_i$ such that for each $i, k$, we have $\p{\delta_{i,k} = 1} = \frac{1}{\lceil\kappa\rceil} = 1 - \p{\delta_{i,k} = 0}$, moreover $\sum_{k=1}^{\lceil\kappa\rceil} \delta_{i,k} = 1$ for each $i$, and the variables $\delta_{1,k},\ldots,\delta_{n,k}$ are independent for each $k$. For instance, we can define them on the probability space $[0,1]^n$ with Lebesgue measure by the formula
\[
\delta_{i,k}(t_1,\ldots,t_n) = \1_{[\frac{k-1}{\lceil\kappa\rceil},\frac{k}{\lceil\kappa\rceil}]}(t_i).
\]
We check that for every $i$ and $k$, we have $\delta_{i,k}X_i \prec_{(1,1)} \lambda Y_i$, so for every $k$ we obtain the comparison $\sum_i \delta_{i,k}X_i \prec_{(\kappa_0,\lambda_0)} \lambda\sum_i Y_i$ and thus
\begin{align*}
\p{\|{\sum_i} X_i\| > 1} = \p{\|{\sum_i\sum_k} \delta_{i,k}X_i\| > 1} &\leq \sum_{k=1}^{\lceil\kappa\rceil} \p{\lceil\kappa\rceil\|{\sum_i} \delta_{i,k}X_i\| > 1} \\
&\leq \lceil\kappa\rceil\kappa_0 \p{\lceil\kappa\rceil\lambda\lambda_0\|{\sum_i} Y_i\| > 1}.
\end{align*}
\end{proof}

The second argument is based on the observation that if for some symmetric independent random vectors $X_i$ and $Y_i$ we have $X \prec_{(\kappa,\lambda)} Y_i$, $i \leq n$, then taking $\delta_i$ to be independent Bernoulli random variables such that $\p{\delta_i = 1} = 1/\kappa$, $\p{\delta_i = 0} = 1- 1/\kappa$ and defining $X_i' = \delta_i X_i$, $Y_i' = \lambda Y_i$ we get $X_i' \prec_{(1,1)} Y_i'$. To obtain $\sum X_i \prec \sum Y_i$, we first apply \eqref{eq:tens} from Corollary \ref{cor:proxy} to the $X_i'$ and $Y_i'$, which gives
\[
\int_0^1 \p{\E_\varepsilon (\|{\sum} \varepsilon_iX'_i\|-1)_+ > t } \dd t \leq \int_0^1 \p{\E_\varepsilon (\|{\sum} \varepsilon_iY_i'\|-1)_+ > t } \dd t.
\] 
Then we bound the right hand side by Lemma \ref{lm:upbd}, but before using Lemma \ref{lm:lowbd} for the lower bound of the left hand side, we would like to pass from the $X_i'$ to $X_i$. This can be achieved if we have an inequality like this one
\[
\mathbb{P}\left(\E_\varepsilon \|{\sum} \varepsilon_i\delta_i X_i \| > u\right) \geq c\mathbb{P}\left(c'\E_\varepsilon \|{\sum} \varepsilon_i X_i \| > u\right), \qquad u > 0,  
\]
where $c$ and $c'$ are some constants. This is possible thanks to a simple lemma which is in the spirit of the Paley-Zygmund inequality.

\begin{lemma}\label{lm:removedelta}
Let $v_1,\ldots,v_n$ be vectors in a separable Banach space $(E,\|\cdot\|)$, $p \in (0,1]$ and let $\delta_1,\ldots,\delta_n$ be independent Bernoulli random variables with $\p{\delta_i = 1} = p$, $\p{\delta_i = 0} = 1-p$. Then
\[
\mathbb{P}_\delta\left( \E_\varepsilon \|{\sum} \varepsilon_i\delta_iv_i\| > 1 \right) \geq \frac{p}{4}\1_{\{\E_\varepsilon \|\sum \varepsilon_i v_i\| > 2/p\}}.
\] 
\end{lemma}

By virtue of this lemma, we can take above $c = \frac{1}{4\kappa}$ and $c' = \frac{1}{2\kappa}$. After passing through Lemma \ref{lm:lowbd} applied to the $X_i$ we conclude that $\sum X_i \prec_{(\kappa',\lambda')} \sum Y_i$ with $\kappa' = 64\alpha^{-1}\kappa$ and $\lambda' = 2(1+\alpha)\kappa\lambda$ for every $\alpha \in (0,1]$. We finish this section by showing the lemma.

\begin{proof}[Proof of Lemma \ref{lm:removedelta}]
Obviously we can assume that $\E_\varepsilon \|\sum \varepsilon_i v_i\| > 2/p$ since otherwise there is nothing to prove. By Jensen's inequality,
\[
\E_{\delta,\varepsilon}\|{\sum} \delta_i\varepsilon_i v_i\| \geq \E_{\varepsilon}\|\E_\delta{\sum} \delta_i\varepsilon_i v_i\| = p\E_{\varepsilon}\|{\sum_i} \varepsilon_i v_i\|,
\]
thus $\frac{1}{2}\E_{\delta}\E_\varepsilon \|\sum \delta_i\varepsilon_i v_i\| \geq \frac{p}{2}\E_\varepsilon \|\sum \varepsilon_i v_i\| > 1 $, so
\begin{align*}
\mathbb{P}_\delta\left( \E_\varepsilon \|{\sum} \varepsilon_i\delta_iv_i\| > 1 \right) &\geq \mathbb{P}_\delta\left( \E_\varepsilon \|{\sum} \varepsilon_i\delta_iv_i\| > \frac{1}{2}\E_{\delta}\E_\varepsilon \|{\sum} \delta_i\varepsilon_i v_i\| \right) \\
&\geq \frac{1}{4}\frac{\left(\E_{\delta}\E_\varepsilon \|\sum \delta_i\varepsilon_i v_i\|\right)^2}{\E_{\delta}\left(\E_\varepsilon \|\sum \delta_i\varepsilon_i v_i\|\right)^2},
\end{align*}
where in the last estimate we used the Paley-Zygmund inequality.
Using the contraction principle \eqref{eq:contr} we obtain ($\varepsilon_i'$ denote independent copies of $\varepsilon_i$)
\begin{align*}
\E_\delta\left( \E_\varepsilon \|{\sum} \delta_i\varepsilon_i v_i\| \right)^2 &= \E_\delta\left( \E_\varepsilon \|{\sum} \delta_i\varepsilon_i v_i\|\E_{\varepsilon'} \|{\sum} \delta_i\varepsilon_i' v_i\|\right) \\
&\leq \left(\E_\varepsilon \|{\sum} \varepsilon_i v_i\|\right)\left(\E_{\delta,\varepsilon} \|{\sum} \delta_i\varepsilon_i v_i\|\right) \leq \frac{1}{p}\left(\E_{\delta,\varepsilon} \|{\sum} \delta_i\varepsilon_i v_i\|\right)^2.
\end{align*}
This combined with the previous inequality finishes the proof.
\end{proof}

\section{Proof of Theorem \ref{thm:WBtensor}}

Suppose $X_1,\ldots,X_n$ are independent symmetric random vectors and each satisfies $WB(C,\delta,\theta)$. Let $S_n = X_1 + \ldots + X_n$. Fix a continuous norm $\|\cdot\|$. We would like to show that 
\[
\p{\|S_n\| > \lambda} \leq C'\lambda^{-\delta}\p{\|S_n\|>1}, \qquad \lambda \geq 1,
\]
provided that $\p{\|S_n\|>1} < \theta'$. (We shall find the values of the constants $C'$ and $\theta'$ as the argument goes along.) First observe that if $\theta' \leq \theta/2$, then by \eqref{eq:Xn*} we also have that
\[
\p{\|X_j\| > 1} \leq \p{X_n^*>1} \leq 2\p{\|S_n\|>1} < 2\theta' \leq \theta, \]
where $X_n^* = \max_{j \leq n} \|X_j\|$. 
This will let us use the WB inequality for $X_j$, $j = 1,\ldots,n$.

Let $p_k=\p{\|S_n\| > 3^k}$ for $k = 0,1,\ldots$. Our first goal is to establish that $p_k \leq C'\cdot 3^{-\delta k}p_0$, assuming $p_0 \leq \theta'$. Then, possibly increasing $C'$ we will get that $\p{\|S_n\| > \lambda} \leq C'\lambda^{-\delta}\p{\|S_n\|>1}$ for every $\lambda \geq 1$. We begin with deriving a recursive inequality for the $p_k$. Fix $k \geq 1$. By \eqref{eq:Levy} - \eqref{eq:sumX_iuppbd} and the union bound we obtain
\begin{align*}
p_k = \p{\|S_n\| > 3\cdot3^{k-1}} &\leq \p{X_n^* > 3^{k-1}} + 2\p{S_n^* > 3^{k-1}}\p{\|S_n\| > 3^{k-1}} \\
&\leq \sum_{j=1}^n\p{\|X_j\|> 3^{k-1}} + 4p_{k-1}^2 \\
&\leq C\cdot 3^{-\delta(k-1)}\sum_{j=1}^n\p{\|X_j\| > 1} + 4p_{k-1}^2 \\
&\leq C\cdot 3^{-\delta(k-1)}\frac{2p_0}{1-2p_0} + 4p_{k-1}^2. 
\end{align*}
If we assume additionally that $\theta' \leq 1/3$, then $\frac{1}{1-2p_0}\leq 3$, so
\[
p_k \leq 6C\cdot 3^{-\delta(k-1)}p_0 + 4p_{k-1}^2, \qquad k \geq 1.\]
Let us prove inductively that $p_k \leq (12\cdot3^{\delta}C)\cdot 3^{-k\delta}p_0$, $k \geq 0$. For $k = 0$ this is obvious. Suppose it holds for $k-1$, for some $k\geq 1$. By the recursive inequality,
\begin{align*}
p_k &\leq 6C\cdot 3^{-\delta(k-1)}p_0 + 4(12\cdot3^\delta C)^2\cdot 3^{-2\delta(k-1)}p_0^2 \\
&= (12\cdot3^\delta C)\cdot 3^{-k\delta}p_0\cdot\left(\frac{1}{2} + 48C\cdot 3^{-\delta k + 3\delta} p_0\right)\\ 
&\leq (12\cdot3^\delta C)\cdot 3^{-k\delta}p_0\cdot\left(\frac{1}{2} + 48C\cdot 3^{2\delta} \theta'\right)
\end{align*}
and we get the inductive assertion as long as $\theta' \leq (96C\cdot 9^\delta)^{-1}$. Therefore we set $\theta' = \min\{\theta/2,(96C\cdot 9^\delta)^{-1}\}$. Then, as we have shown,
\[
\p{\|S_n\|>\lambda} \leq (12\cdot3^{\delta}C)\cdot \lambda^{-\delta}\p{\|S_n\|>1},\]
for $\lambda = 3^k$, $k \geq 0$. It remains to extend this to any $\lambda \geq 1$. If $1 \leq \lambda < 3$, then trivially
\[
\p{\|S_n\| > \lambda} \leq \p{\|S_n\| > 1} \leq 3^{\delta}\lambda^{-\delta}\p{\|S_n\| > 1}.\]
If $3^k \leq \lambda < 3^{k+1}$ for some $k \geq 1$, we get
\begin{align*}
\p{\|S_n\| > \lambda} &\leq \p{\|S_n\|>3^k} \leq (12\cdot3^{\delta}C)\cdot 3^{-k\delta}\p{\|S_n\|>1} \\
&\leq (12\cdot3^{2\delta}C)\cdot \lambda^{-\delta}\p{\|S_n\|>1}.
\end{align*} 
We set $C' = 12\cdot 9^{\delta}C$ and the proof is complete. \hfill$\square$

\section{Proof of Theorem \ref{thm:schur}}

Since the sequence $a$ is majorised by $b$, there are nonnegative weights $\lambda_\sigma$ adding up to $1$ indexed by all permutations $\sigma$ of the $n$-element set $\{1,\ldots,n\}$ such that $a = \sum_\sigma \lambda_\sigma b_\sigma$, where $b_\sigma = (b_{\sigma(1)},\ldots,b_{\sigma(n)})$ is the sequence $b$ permuted according to $\sigma$. It easily follows that for every convex function $\varphi\colon E\to \R$ defined on the Banach space $E$ the $X_i$ take values in, we have
\[
\E \varphi\left(\sum a_iX_i\right) = \E \varphi\left(\sum_i\sum_\sigma \lambda_\sigma b_{\sigma(i)}X_i\right) \leq \sum_\sigma\lambda_\sigma \E\varphi\left(\sum_i b_{\sigma(i)}X_i\right) = \E\varphi\left(\sum b_iX_i\right)  
\]
(provided the expectations exist).

Notice that since each $b_iX_i$ satisfies $\text{WB}(C,\delta,\theta)$, by Theorem \ref{thm:WBtensor}, the sum $\sum b_iX_i$ satisfies $\text{WB}(C',\delta,\theta')$, where $C' = 12\cdot 9^{\delta}C$ and $\theta' = \min\{\theta/2,(96C\cdot 9^\delta)^{-1}\}$. 

Let $\|\cdot\|$ be a continuous norm on $E$. Denote $S_a = \|\sum a_iX_i\|$ and $S_b = \|\sum b_iX_i\|$. We want to show that $\p{S_a > 1} \leq \kappa \p{\lambda S_b > 1}$. If $\p{S_b > 1} \geq \theta'$, then we trivially get
\[
\p{S_a > 1} \leq 1 = \frac{1}{\theta'}\theta' \leq \frac{1}{\theta'}\p{S_b > 1}.
\]
Suppose that $\p{S_b > 1} < \theta'$. Using the initial observation for $\varphi(x) = (\|x\|-1)_+$ we get
\[
\p{S_a > 2} = \p{(S_a-1)_+ > 1} \leq \E(S_a-1)_+ \leq \E(S_b-1)_+.
\]
By the weak Borell inequality for $S_b$ we get
\[
\E(S_b - 1)_+ = \int_1^\infty \p{S_b > \lambda} \dd \lambda \leq \int_1^\infty C'\lambda^{-\delta}\p{S_b > 1}\dd \lambda = \frac{C'}{\delta - 1}\p{S_b > 1}.
\]
In summary, we have showed that for any continuous norm $\|\cdot\|$ on $E$, we have
\[
\p{\|\sum a_iX_i\| > 2} \leq \kappa \p{\|\sum b_iX_i\| > 1}
\]
with $\kappa = \max\{\frac{1}{\theta'},\frac{C'}{\delta-1}\}$. Rescaling the norm finishes the proof. \hfill$\square$

\subsection*{Acknowledgments}

The first named author was partially supported by Narodowe Centrum Nauki (Poland), grant no.
2016/21/B/ST1/01489.

This material is partially based upon work supported by the NSF under Grant No. 1440140, while the second named author was in residence at the MSRI in Berkeley, California, during the fall semester of 2017. He was also partially supported by the Simons Foundation. This work was initiated when he visited the University of Warsaw in April 2017. He is immensely grateful for their hospitality.

\end{document}